\documentclass{amsart}

\newcommand\Q{\mathbb{Q}}	\newcommand\R{\mathbb{R}}	\newcommand\C{\mathbb{C}}	\newcommand\bbP{\mathbb{P}}
	
	\newcommand\M{\mathcal{M}}

\newcommand\logder{\ell}

\newcommand\ord{\operatorname{ord}}	\newcommand\disc{\operatorname{disc}}

\newcommand\arch{{\textup{arch}}}	\newcommand\na{{\textup{na}}}

\newtheorem{theorem}{Theorem}[section]

\newtheorem{lemma}[theorem]{Lemma}

\theoremstyle{definition}\newtheorem{definition}[theorem]{Definition}
\theoremstyle{remark}\newtheorem{remark}[theorem]{Remark}
\newtheorem{conjecture}[theorem]{Conjecture}
\numberwithin{equation}{section}

\begin{document}

\title{The ABC Theorem for Meromorphic Functions}

\author{Machiel van Frankenhuijsen}

\address{Utah Valley University,
Department of Mathematics,
800 West University Parkway,
Orem, Utah 84058-5999}
\email{vanframa@uvu.edu}


\date{May 12, 2008}


\begin{abstract}
Using a `height-to-radical' identity,
 we define the archimedean contribution to the radical,
$r_\arch$,
and we give a new proof of the abc theorem for the field of meromorphic functions.
The first step of the proof is completely formal and yields that the height is bounded by the radical,
$h\leq r$,
where $r=r_\na+r_\arch$ is the radical completed with the archimedean contribution.
The second step is analytic in nature and uses the lemma on the logarithmic derivative to derive a bound for $r_\arch$.
\end{abstract}

\maketitle

\section{Introduction}
\label{S: intro}

Although the abc theorem for meromorphic functions was first formulated and proved in \cite{thesis},
it has essentially been known since the beginning of the twentieth century.
For example,
it is an easy consequence of Nevanlinna's second main theorem.
Since Mason \cite{Mason},
the point of view has changed,
especially when it became known that the abc conjecture for numbers implies an effective version of Mordell's conjecture
(see \cite{Elkies,abcRM,abcvhi,abcrvhi};
the proofs of Faltings and Vojta \cite{F1,F2,VojtaMordell} are ineffective).

Here,
we prove this theorem again,
organizing the proof in two steps to reveal the main structure.
In the first step,
we use a `height-to-radical' identity to obtain in a completely formal way $h(\rho)\leq r(\rho)$ for every radius $\rho\geq1$,
where $h$ and $r$ are respectively the (logarithmic) height and radical of an abc sum $a(z)+b(z)+c(z)=0$
of meromorphic functions.
We will see how to define a natural contribution of the archimedean valuations to the radical,
$r(\rho)=r_\na(\rho)+r_\arch(\rho)$.
Then in the second step,
we apply the lemma on the logarithmic derivative to obtain the inequality $r_\arch(\rho)\leq2\log h(\rho)+O(\log\log h(\rho))$
for every $\rho\geq1$ outside an open subset $E$ of $[1,\infty)$ of finite total length.\medskip

To explain the motivation from number theory,
let $k$ be a number field.
The valuations of $k$ satisfy the Artin--Whaples sum formula:
for every $x\in k^*$,
\begin{gather}\label{Artin-Whaples}
\sum_vv(x)=0.
\end{gather}
Here,
 we fix the archimedean valuations so that $v(2)=[k_v:\R]\log2$,
and the nonarchimedean valuations so that $v(p)=-[k_v:\Q_p]\log p$,
where $p$ is the rational prime such that $v(p)<0$.
Note that we define our valuations with a minus sign as compared to \cite{Serre}.

For a point $P=(a:b:c)\in\bbP^2(k)$,
we define the contribution of the valuation $v$ to the height by
\begin{gather*}
h_v(a,b,c)=\begin{cases}
\max\{v(a),v(b),v(c)\}				&\text{if $v$ is nonarchimedean,}\\
\frac{[k_v:\R]}2\log\bigl(|a|^{2}+|b|^{2}+|c|^{2}\bigr)&\text{if $v$ is archimedean},
\end{cases}
\end{gather*}
where we write $|a|=\sqrt{a\bar a}$ for the usual absolute value of a complex number.
The {\em height\/} of $P$ is defined by
\begin{gather*}
h(P)=\sum_vh_v(a,b,c).
\end{gather*}
By (\ref{Artin-Whaples}),
the height does not depend on the choice of coordinates for $P$,
even though the local contributions do depend on this choice.

\begin{remark}
The height depends on the number field,
but the `canonical height' $h(P)/[k:\Q]$ does not depend on $k$.
On the other hand,
for the radical,
to be defined below,
there is no good definition of a `canonical radical'.
\end{remark}

The {\em degree\/} of a valuation is defined by
\begin{gather*}
\deg(v)=\begin{cases}
\log\#k(v)	&\text{if $v$ is nonarchimedean,}\\
0		&\text{if $v$ is archimedean,}
\end{cases}
\end{gather*}
where $k(v)$ denotes the residue class field of $v$.
The contribution of the valuation $v$ to the radical is defined by
\begin{gather*}
r_{v}(P)=\begin{cases}
0	&\text{if }v(a)=v(b)=v(c),\\
\deg(v)	&\text{otherwise,}
\end{cases}
\end{gather*}
which is clearly independent of the choice of coordinates for $P$.
We define the {\em incomplete radical\/} of $P$ by
\begin{gather*}
r_\na(P)=\sum_vr_{v}(P).
\end{gather*}
Thus the radical has a contribution from every nonarchimedean valuation where the orders of $a$,
$b$ and $c$ are not all equal to each other.
In particular,
$r_\na(P)=\infty$ if one of the coordinates of $P$ vanishes.
It has no contribution from the archimedean valuations,
since we have put $\deg(v)=0$ for these valuations.

\begin{conjecture}[ABC conjecture with type $\psi$ \cite{Oesterle,abcrvhi,VojtaABC}]
\label{C: abc}
There exists a function $\psi$ with $\psi (h)=o(h)$ such that for every point $P=(a:b:c)\in\bbP^2( k)$ on the line $a+b+c=0$,
\begin{gather*}
h(P)\leq r_\na(P)+\psi(h(P))+\log|\disc(k)|.
\end{gather*}
One may further conjecture that $\psi/[k:\Q]$ is independent of the number field $k$.
\end{conjecture}

\begin{remark}
This conjecture can be interpreted as a weak form of Hurwitz' formula.
See Remark \ref{Hurwitz} and \cite{Smirnov}.
\end{remark}

The abc conjecture is known with $\psi(h)=h-C\log h$ for a constant $C$ (this function is not $o(h)$) \cite{ST,SYu},
and if the abc conjecture holds for some function $\psi$,
then we must have $\psi(h)\geq 6.07{\sqrt{h}}/{\log h}$ \cite{ST,lowerbound}.
Numerical evidence suggests that the abc conjecture may hold with $\psi(h)=4[k:\Q]\sqrt h$.
Thus we obtain the stronger conjecture
\begin{quote}
For every point $P=(a:b:c)\in\bbP^2( k)$ on the line $a+b+c=0$,
$$
h(P)\leq r_\na(P)+4[k:\Q]\sqrt{h(P)}+\log|\disc(k)|.
$$
\end{quote}
In Section \ref{S: conclusion},
we give a possible interpretation of the error term $\psi(h)$.\medskip

It is rather surprising that this conjecture implies an effective version of both Vojta's height inequality \cite{abcvhi} and the radicalized Vojta heigh inequality \cite{abcrvhi} (see \cite[Conjectures 2.1 and 2.3]{VojtaABC}).
In particular,
it implies Roth's theorem with an effective error term and an effective version of Mordell's conjecture.
These implications use the construction of a Bely\u\i\ function,
for which there is no analogue for function fields or the field of meromorphic functions.

\section{The Height-to-Radical Identity}

Let $\M$ be the field of meromorphic functions on $\C$,
and let $(f)^\logder$ denote the logarithmic derivative,
\begin{gather*}
(f)^\logder(z)=\frac{f'(z)}{f(z)}.
\end{gather*}
It is easily verified that $(f)^\logder(z)$ is a meromorphic function with only simple poles.

We call a point $P(z)=(a(z):b(z):c(z))\in\bbP^2(\M)$ {\em nonconstant\/} if at least one of the functions $a/b$,
$b/c$ or $c/a$ is not constant.

\begin{lemma}[Height-to-Radical Identity]
For a nonconstant point $P=(a:b:c)$ on the line $a+b+c=0$ in $\bbP^2(\M),$
we have the identity
\begin{gather}\label{hr}
(a:b:c)=\left((b/c)^\logder:(c/a)^\logder:(a/b)^\logder\right).
\end{gather}
\end{lemma}

\begin{proof}
The right-hand side does not change if we replace $(a,b,c)$ by $(a/c,b/c,1)$ (note that $P$ is constant if $c=0$).
Hence we need to check that for $f+g+1=0$,
$$
(f:g:1)=\left({g'}/{g}:-{f'}/{f}:{f'}/f-{g'}/{g}\right).
$$
Now $f'+g'=0$ and $g'$ does not vanish identically (since $P$ is nonconstant),
hence this identity follows after multiplying the left-hand side by $g'$ and the right-hand side by $fg$.
\end{proof}

\begin{remark}
The height-to-radical identity provides a canonical way to choose meromorphic coordinates for a point $(a:b:c)$ on the line
 $a+b+c=0$.
Indeed,
replacing $(a,b,c)$ by $(\lambda a,\lambda b,\lambda c)$,
for a meromorphic function $\lambda(z)$,
 does not change the right-hand side of \eqref{hr}.
\end{remark}

\section{The Valuations of the Field of Meromorphic Functions}
\label{S: Nevan}

The valuations of $\M$ satisfy the Poisson--Jensen formula
\begin{gather}\label{PJ}
\sum_{|x|<\rho}v_x(f,\rho)+\int_{|z|=\rho}v_z(f,\rho)\frac{dz}{2\pi iz}-v_\infty(f)=0,
\end{gather}
where the nonarchimedean valuations are parametrized by $x$ with $|x|<\rho$,
$$
v_x(f,\rho)=\begin{cases}
-\ord(f,x)\log\frac\rho{|x|}	&\text{for }0<|x|<\rho,\\
-\ord(f,0)\log\rho		&\text{for }x=0;
\end{cases}
$$
the archimedean valuations are parametrized by $z$ on the circle of radius $\rho$,
$$
v_z(f,\rho)=\log|f(z)|,
$$
and $v_\infty(f)$ is the absolute value of the first coefficient in the Laurent series of $f$ around $0$:
for $f(z)=f_nz^n+f_{n+1}z^{n+1}+\dots$,
$$
v_\infty(f)=\log|f_n|,\quad\text{where $n=\ord(f,0)$ and }f_n=\lim_{z\to0}f(z)z^{-n}.
$$
We count this function among the archimedean valuations,
even though,
strictly speaking,
 it is not a valuation.
Note that $v_0(f,\rho)$ is only a valuation for $\rho\geq1$.

For a point $P=(a:b:c)\in\bbP^2(\M)$,
the local contributions to the height are
\begin{align*}
&h_x(a,b,c)(\rho)=\max\left\{v_x(a,\rho),v_x(b,\rho),v_x(c,\rho)\right\},\quad\text{ for }v_x,\ |x|<\rho,\\
&h_z(a,b,c)(\rho)=\log\sqrt{|a(z)|^2+|b(z)|^2+|c(z)|^2},			\quad\text{ for }v_z,\ |z|=\rho,\\
&h_\infty(a,b,c)=\log\sqrt{|a_m|^2+|b_m|^2+|c_m|^2},			\quad\text{ for }v_\infty,
\end{align*}
where $m=\min\{\ord(a,0),\ord(b,0),\ord(c,0)\}$.
Then the {\em height\/} of $P$ is defined by
\begin{gather}\label{height}
h(P,\rho)=\sum_{|x|<\rho}h_x(a,b,c)(\rho)+\int_{|z|=\rho}h_z(a,b,c)(\rho)\frac{dz}{2\pi iz}-h_\infty(a,b,c).
\end{gather}
By the Poisson--Jensen formula (\ref{PJ}),
the height does not depend on the choice of coordinates for $P$.

The {\em degree\/} of a valuation is defined by
\begin{align*}
&\deg(v_x,\rho)=\begin{cases}
\log\frac\rho{|x|}&\text{ for }0<|x|<\rho,\\
\log\rho	&\text{ for }x=0,
\end{cases}\\
&\deg(v_z,\rho)=0\quad\text{ for }|z|=\rho,\\
&\deg(v_\infty,\rho)=0.
\end{align*}
Note that $\deg(v_0,\rho)\geq0$ only for $\rho\geq1$.
All other nonarchimedean valuations have a positive degree,
and the archimedean valuations have vanishing degree.

The local contributions to the radical are defined by
\begin{gather*}
r_v(P,\rho)=\begin{cases}
0	&\text{if }v(a)=v(b)=v(c),\\
\deg(v,\rho)	&\text{otherwise,}
\end{cases}
\end{gather*}
independent of the choice of coordinates for $P$.
Thus the archimedean valuations $v_z$ and $v_\infty$ do not contribute to the radical.
We define the {\em incomplete radical\/} of a point $P=(a:b:c)\in\bbP^2(\M)$ by
\begin{gather*}
r_\na(P,\rho)=\sum_vr_{v}(P,\rho).
\end{gather*}
Except if one of the coordinates of $P$ vanishes identically,
this sum is finite for every $\rho>0$ since the set of zeros and poles of $a$,
$b$ and $c$ form a discrete set.

\section{The Formal ABC Theorem for Meromorphic Functions}
\label{S: formal abc}

\begin{theorem}[\cite{thesis}]\label{L: formal abc}
Let\/ $P=(a:b:c)$ be a nonconstant point in\/ $\bbP^2(\M)$ such that\/ $a+b+c=0$.
Then,
 for every\/ $\rho\geq1,$
\begin{gather*}
h(P,\rho)\leq r_\na(P,\rho)+\int_{|z|=\rho}h_z\bigl({\textstyle
(\frac bc)^\logder,(\frac ca)^\logder,(\frac ab)^\logder}\bigr)(\rho)\frac{dz}{2\pi iz}
-h_\infty\bigl(\textstyle{(\frac bc)^\logder,(\frac ca)^\logder,(\frac ab)^\logder}\bigr).
\end{gather*}
\end{theorem}

\begin{proof}
Since $P=\left((b/c)^\logder:(c/a)^\logder:(a/b)^\logder\right)$ by the height-to-radical identity,
we use these coordinates to compute the height of $P$.
Let $x$ be a point with $|x|<\rho$ such that $\ord(a,x)>\ord(b,x)=\ord(c,x)$.
Then $c/a$ has a pole and $a/b$ has a zero at $x$,
hence $(c/a)^\logder$ and $(a/b)^\logder$ have a simple pole at $x$.
Moreover,
$(b/c)^\logder$ has no pole at $x$.
Therefore,
 $v_x$ contributes equally to the height and the radical,
and the same holds for all points where either $b$ or $c$ has a larger order of vanishing.
We obtain
$$
h_x(P,\rho)=r_x(P,\rho)\text{ whenever }v_x(a),v_x(b)\text{ and }v_x(c)\text{ are not all equal.}
$$

For the points $x$ with $|x|<\rho$ where $v_x(a)=v_x(b)=v_x(c)$,
none of the coordinates of $P$ has a pole,
hence
$$
h_x(P,\rho)\leq0=r_x(P,\rho).
$$
(For $x=0$,
we need that $\rho\geq1$.)
Adding to these contributions the archimedean contributions to the height yields the theorem.
\end{proof}

\begin{definition}
We define the {\em archimedean contribution\/} to the radical of a nonconstant point $P=(a:b:c)$
on the line $a+b+c=0$ in $\bbP^2(\M)$ by
$$
r_\arch(P,\rho)=\int_{|z|=\rho}h_z\bigl((b/c)^\logder,(c/a)^\logder,(a/b)^\logder\bigr)(\rho)\frac{dz}{2\pi iz}
-h_\infty\bigl((b/c)^\logder,(c/a)^\logder,(a/b)^\logder\bigr).
$$
The {\em completed radical\/} is defined by $r(P,\rho)=r_\na(P,\rho)+r_\arch(P,\rho).$
\end{definition}

Note that $r_\arch(P,\rho)$ does not depend on the choice of coordinates for $P$.
With these definitions,
Theorem \ref{L: formal abc} reads

\begin{theorem}[Formal ABC]\label{T: formal abc}
For every nonconstant point\/ $P=(a:b:c)$ on the line\/ $a+b+c=0$ in\/ $\bbP^2(\M),$
\begin{gather*}
h(P,\rho)\leq r(P,\rho),
\end{gather*}
for every\/ $\rho\geq1$.
\end{theorem}

\begin{remark}\label{Hurwitz}
(See \cite{abcRM,abcvhi}.)
By Hurwitz' formula for a function $f\colon C\rightarrow\bbP^1$ from an
algebraic curve $C$ of genus $g$ to the projective line,
\begin{align*}
\deg f&=\#f^{-1}\{0,1,\infty\}+2g-2-\sum_{x\colon f(x)\neq0,1,\infty}(\ord(f,x)-1)\\
&\leq\#f^{-1}\{0,1,\infty\}+2g-2.
\end{align*}
We interpret Theorem \ref{T: formal abc} as a weak form of Hurwitz' formula,
where $\deg f$ is the height of $(f:1-f:-1)$,
and $\#f^{-1}\{0,1,\infty\}$ is the radical of $(f:1-f:-1)$.
\end{remark}

\section{The ABC Theorem for Meromorphic Functions}
\label{S: abc}

\begin{lemma}\label{L: bound}
For a nonconstant point\/ $P=(a:b:c)$ on the line\/ $a+b+c=0$ in $\bbP^2(\M),$
there exists an `exceptional set'\/ $E\subset(0,\infty)$ of finite total length such that
$$
r_\arch(P,\rho)\leq2\log h(P,\rho)+O(\log\log h(P,\rho))\quad\text{for all }\rho>0,\ \rho\not\in E.
$$
\end{lemma}

\begin{proof}
We first recall the lemma on the logarithmic derivative \cite[Theorem 6.1, p.\ 48]{LC}.
Let the proximity function of a meromorphic function $f$ be defined by
$$
m(f,\rho)=\int_{|z|=\rho}\log\sqrt{1+|f(z)|^2}\,\frac{dz}{2\pi iz},
$$
and the height by $h(f,\rho)=h((f:1:0),\rho)$,
as defined in (\ref{height}).
Then,
for a nonconstant meromorphic function $f$,
there exists a subset $E\subset(0,\infty)$ of finite total length such that
$$
m((f)^\logder,\rho)\leq\log h(f,\rho)+O(\log\log h(f,\rho))\quad\text{for all }\rho>0,\ \rho\not\in E.
$$

Choose now $a$,
$b$ and $c$ such that $(a,b,c)=\left((b/c)^\logder,(c/a)^\logder,(a/b)^\logder\right)$.
Since $c=-a-b$,
we have
$|a|^2+|b|^2+|c|^2\leq2(1+|a|^2)(1+|b|^2)$.
Hence
$$
\int_{|z|=\rho}h_z(a,b,c)(\rho)\frac{dz}{2\pi iz}\leq\log\sqrt2+m(a,\rho)+m(b,\rho).
$$
Since $a=(b/c)^\logder$ and $b=(c/a)^\logder$,
we find a subset $E$ of $(0,\infty)$ of finite total length such that
\begin{multline*}
\int_{|z|=\rho}h_z(a,b,c)(\rho)\frac{dz}{2\pi iz}\leq\log h(b/c,\rho)+\log h(c/a,\rho)\\
+O(\log\log h(b/c,\rho)+\log\log h(c/a,\rho))
\end{multline*}
for all $\rho>0$,
$\rho\not\in E$.
Since $h(b/c,\rho)$ and $h(c/a,\rho)\leq h(P,\rho)$,
 the lemma follows.
\end{proof}

\begin{remark}
Using an ultrafilter on $(0,\infty)$,
one can suppress the dependence on $\rho$,
see \cite{Philippon}.
In this formalism,
there is no exceptional set $E$.
\end{remark}

Applying the bound for $r_\arch(P,\rho)$ of Lemma \ref{L: bound},
we obtain from Theorem \ref{T: formal abc},

\begin{theorem}[Function Theoretic ABC]
Let\/ $P=(a:b:c)$ be a nonconstant point on the line\/ $a+b+c=0$ in\/ $\bbP^2(\M)$.
Then there exists an open set\/ $E\subset(1,\infty)$ of finite total length such that
$$
h(P,\rho)\leq r_\na(P,\rho)+2\log h(P,\rho)+O(\log\log h(P,\rho)),
$$
for every $\rho\geq1,$
$\rho\not\in E$.
\end{theorem}

\section{Conclusion}
\label{S: conclusion}

From Lemma \ref{L: bound},
we see that the term $2\log h(P,\rho)+O(\log\log h(P,\rho))$ is a bound for the archimedean contribution to the radical.
This suggests that the function $\psi$ in Conjecture \ref{C: abc} should be viewed as
a bound for the archimedean contribution to the radical,
and not as this contribution itself.
Therefore,
we propose for a number field $k$,
\begin{itemize}
\item[1.]
to define a completed radical $r(P)=r_\na(P)+r_\arch(P)$;
\item[2.]
to show that $h(P)\leq r(P)$ for every point on the line $a+b+c=0$ in $\bbP^1(k)$;
\item[3.]
to obtain a bound of the type $r_\arch(P)\leq\psi(h(P))$.
\end{itemize}
Presumably,
step 2 will follow from formal properties of the definition in step 1.
Step 3 will be the hard step.
The strongest possible result will have $\psi(h)=O(\sqrt h/\log h)$,
but even a result where $\psi(h)=(1-1/C)h$ for some $C>0$ would be very interesting,
since it would imply $h\leq Cr$.
This would settle,
for example,
Fermat's Last Theorem for all exponents greater than $3C$.

\bibliographystyle{amsart}

\end{document}